\documentclass[12pt]{article}

\newcommand{\Pos}{\mbox{Pos}}
\newcommand{\Tr}{\mbox{Tr}}

\usepackage[latin1]{inputenc}
\usepackage{amsfonts,epsfig,graphicx}
\usepackage{amsmath,amssymb,amsthm}
\usepackage{epsf}
\usepackage{graphics}
\usepackage{amsfonts}
\usepackage{amsmath}
\usepackage{amssymb}
\usepackage[usenames,dvipsnames]{color}
\usepackage{graphicx}
\usepackage{times}

\usepackage{latexsym}

\usepackage{mathbbol}
\usepackage{bbm}
\usepackage[colorlinks,citecolor=blue]{hyperref}
\usepackage{tikz}
\usetikzlibrary{calc,positioning}
\usetikzlibrary{patterns}
\usepackage{bbm}
\usepackage{float}
\usepackage{enumitem}
\usepackage{comment}

\usepackage{MnSymbol}

\usepackage[american voltages, american currents]{circuitikz}

\usepackage{tikz,pgfplots}
\usepgfplotslibrary{polar}
\usepgflibrary{shapes.geometric}
\usetikzlibrary{calc}

\usepackage{braket}

\pgfplotsset{my style/.append style={axis x line=middle, axis y line=middle, xlabel={$x$}, ylabel={$y$}, axis equal }}

\pgfplotsset{graphstyle01/.append style={axis x line=middle, axis y line=middle, xlabel={$x$}, ylabel={$y$}, axis equal }}

\newtheorem{lemma}{Lemma}

\newtheorem{remark}{Remark}
\newtheorem{corollary}{Corollary}

\newcommand{\mbbC}{\mathbb{C}}
\newcommand{\mbbR}{\mathbb{R}}

\newcommand{\mcA}{\mathcal{A}}
\newcommand{\mcU}{\mathcal{U}}
\newcommand{\mcV}{\mathcal{V}}
\newcommand{\mcX}{\mathcal{X}}
\newcommand{\mcY}{\mathcal{Y}}
\newcommand{\mcZ}{\mathcal{Z}}

\newcommand{\tilu}{\tilde{u}}
\newcommand{\tilv}{\tilde{v}}
\newcommand{\tilU}{\tilde{U}}
\newcommand{\tilV}{\tilde{V}}

\begin{document}

\begin{center}
{\Large \bf Quantum advantage in decentralized control
of POMDPs: A control-theoretic view of the Mermin-Peres square}\\
~\\
Venkat Anantharam\\
{\em 30 December 2024}\\
{\sl EECS Department}\\
{\sl University of California}\\
{\sl Berkeley, CA 94720, U.S.A.}\\
~\\
{\em (Dedicated to the memory of Pravin P. Varaiya)}
~\\
~\\
{\bf Abstract}
\end{center}

Consider a decentralized partially-observed Markov decision 
problem (POMDP) with multiple cooperative agents aiming to maximize
a long-term-average reward criterion. We observe that the
availability, at a fixed rate, of entangled states 
of a product quantum system between the agents, where
each agent has access to one of the component systems, can 
result in strictly improved performance
even compared to the scenario where common randomness is provided to the agents, i.e. there is a 
quantum advantage in decentralized control. 
This observation comes from a simple reinterpretation of the conclusions of the well-known Mermin-Peres square, which underpins the Mermin-Peres game. While quantum advantage has been demonstrated earlier in one-shot team problems of this kind, it is notable that there are examples where there is a quantum advantage for the one-shot criterion but it disappears in the dynamical scenario. The presence of a quantum advantage in dynamical scenarios is thus seen to be a novel finding relative to the current state of knowledge about the achievable performance in decentralized control problems.

This paper is dedicated to the memory of Pravin P. Varaiya.
%\maketitle

\section{Introduction}      \label{sec:introduction}

Consider a pair of agents, Alice and Bob, where Alice has access
to the random variable $A$ and Bob has access to the random 
variable $B$, with $(A,B)$ having some joint distribution which,
for simplicity, we assume is on a finite set.
It is of interest to study the set of all joint
probability distributions $p(x,y,a,b)$ 
%that can be created 
where
$X$ is a finite random variable created by Alice without access to $B$ and $Y$ is a finite random variable is created
by Bob without access to $A$. Technically, the concept of ``without access"
corresponds to the so-called ``no-signaling" 
conditions $I(X;B|A) = 0$ and $I(Y;A|B) = 0$ on the
respective conditional
mutual information terms.
The availability of an entangled state of a product quantum system between Alice and Bob, i.e. where Alice has access to the first component of the product state and Bob has access to the second component,
allows for a larger
class of such joint distributions to be created than those
that can be created even with
unlimited common randomness provided to Alice and Bob.
This quantum advantage has been widely 
studied in the foundations of quantum mechanics, associated
with the topic of Bell inequalities; for an overview of some
of this literature, see e.g. \cite{AMO08,GKetc18}.
Since we think of Alice and Bob as working together 
to explore the space of all possible $p(x,y,a,b)$ for a given
$p(a,b)$, for a control-theorist this problem belongs to the 
general subject of team theory, see e.g.
\cite{YukBas24}. For an interesting perspective on the origins
of team theory in control see \cite{BasBan89}.

Recently, several works have begun to appear exploring the 
value of such a quantum advantage in the framework of problems
of decentralized control and game theory, see e.g.
\cite{DK22, DKcdc23, DK1binary23, DK2binary23, Milch24, SYsurvey22}.
See also \cite{AB07}
for an earlier work suggesting both the need to develop a 
theory of
games between teams and the importance in distributed control
of recognizing the gap between joint distributions satisfying
the no-signaling condition and those achievable by common randomness
between the individual decision-makers. This work can be considered as
belonging to this general stream of ideas. For recent works
building a theory of games between teams, see e.g. 
\cite{HBY23, KNM22, SSY24}.

In this work we consider a decentralized partially-observed Markov decision 
problem (decentralized POMDP) with multiple cooperative agents aiming to maximize
a long-term-average reward criterion. 
For simplicity, we focus on the case with two agents, who 
we might as well call Alice and Bob. We observe that the
availability of a stream of entangled product quantum states between the agents can 
result in strictly improved performance, i.e. there is a 
quantum advantage in a decentralized control. 
As opposed to the earlier works exploring quantum advantage in 
decentralized control and team theory referenced above, our framework is 
dynamical, and the quantum advantage is established relative
to all possible adapted classical strategies in this dynamical 
framework. We also give an example where there is a quantum advantage in the static (one-shot) problem 
of maximizing the expected reward at a given time,
but where there is no quantum advantage in the underlying dynamical problem. Thus the presence of dynamical quantum advantage in decentralized control, demonstrated in this paper, is a genuinely new finding. For static team problems the presence of such a quantum advantage has been established earlier, see e.g. 
\cite[Section 4.2]{DK22},
\cite[Example 5.2]{SYsurvey22}.

Our observation is based on a simple reinterpretation of the
well-known and astonishing example in the theory of quantum information called the Mermin-Peres square, which is used in the
so-called Mermin-Peres game, see e.g. \cite{Aravind04}, \cite[Sec. 3.2.2]{Holevo12}, \cite{Mermin90, Peres90}.
No prior familiarity with quantum information is needed to read this paper, since all the essential quantum mechanical background is rigorously and succinctly 
developed in Appendix \ref{app:quantum} and Appendix
\ref{app:mpsquare}. This paper should therefore be accessible to a broad community of control theorists.

This paper is dedicated to the memory of Pravin P. Varaiya,
who contributed several seminal works to the early development of decentralized control and team theory, 
%such as \cite{SVASsurvey78, VarWal78} 
and who, throughout his career, was fascinated by the 
intricate questions about knowledge arising from decentralized information structures. 

\section{A model for a class of decentralized POMDPs
%partially-observed Markov decision processes 
%with two controllers
}   \label{sec:generalmodel}

Our purpose is to make a qualitative
point about the advantage provided by quantum entanglement in decentralized control.
Therefore, we eschew generality and focus on a simple decentralized 
%partially-observed Markov decision process 
POMDP
model with two agents, Alice and Bob, who are working together to maximize a long-term-average-reward criterion.
Further, we assume that the observations of Alice and Bob at each time are drawn from finite sets, as are their actions. Indeed, we will simply assume that Alice and Bob each see one of the components of a two-component state at each time.

Formally, the state of the system evolves in discrete time in the 
finite set $\mcX \times \mcY$ under the influence of an action pair drawn from the finite set 
$\mcU \times \mcV$. The initial condition is
$(X_0,Y_0)$, possibly random, with Alice observing $X_0$ and Bob observing $Y_0$, and both knowing the initial 
probability distribution.
The one-step transition probabilities at time
$n \ge 0$ are time-homogeneous, given by 
the Markovian kernel
\[
q(x_{n+1}, y_{n+1}| x_n, y_n, u_n, v_n),
\]
i.e. 
\begin{equation}    \label{eq:Markov}
\aligned
 P((X_{n+1}, Y_{n+1}) &= (x_{n+1}, y_{n+1}) |
(X_n, Y_n, U_n, V_n) = (x_n, y_n, u_n, v_n), 
\nonumber \\
%& ~~~~~~~~~ ~~~~~~~~~~~~~~~~~~~~~~~~~~~~~~~~~~~~~~~~~~~~~~ %X_0^{n-1}, Y_0^{n-1}, U_0^{n-1}, V_0^{n-1}, Z)
&~~~~~~~~~~~~~~~~~~~~~~~~~~~~~X_{0:n-1}, Y_{0:n-1}, U_{0:n-1}, V_{0:n-1}, W_{0:n})
\nonumber \\
& ~~~~~~~  = q(x_{n+1}, y_{n+1}| x_n, y_n, u_n, v_n).
\endaligned
\end{equation}
Here, we allow for an unlimited amount of common 
randomness between Alice and Bob, represented
by 
%the random variable $Z$ 
%the sequence of random variables $W_0^{n-1}$
the sequence of random variables $(W_n, n \ge 0)$,
%in eqn. \eqref{eq:Markov}.
%Here $(W_n, n \ge 0)$ 
which are assumed to be independent and can each have
an arbitrary distribution taking values in an arbitrary
complete separable metric space.
%which is assumed to be
%uniformly distributed in $[0,1]$ and 
%independent of the pair $(X_0, Y_0)$. 
Further, $(W_n, n \ge 0)$ is assumed to be independent
of the pair $(X_0, Y_0)$. 
For $n \ge 0$, we think of $W_n$ as being provided to both Alice and Bob at time $n$. 
Alice observes 
$((X_n, W_n), n \ge 0)$ and chooses $(U_n, n \ge 0)$ (causally),
while Bob observes $((Y_n, W_n), n \ge 0)$ and 
gets to causally choose $(V_n, n \ge 0)$. Both agents
know the structure of the one-step transition probabilities. 

Formally,
the control strategy of Alice is given by 
deterministic functions
\begin{equation}    \label{eq:AliceU}
U_n = u_n(X_{0:n}, W_{0:n}), n \ge 0,
\end{equation}
and that of Bob by deterministic functions
\begin{equation}        \label{eq:BobV}
V_n = v_n(Y_{0:n}. W_{0:n}), n \ge 0.
\end{equation}
Let
\[
r: \mcX \times \mcY \times \mcU \times \mcV \to \mbbR,
\]
be some given fixed reward function. The 
shared aim of Alice and Bob is to choose their
strategies to as to maximize the long-term-average reward
\begin{equation}        \label{eq:performance}
\liminf_{N \to \infty}
\frac{1}{N} \sum_{n=0}^{N-1} 
E[ r(X_n, Y_n, U_n, V_n)].
\end{equation}

A decentralized POMDP of the kind described above 
will be characterized via
\[
(\mcX, \mcY, \mcU, \mcV, q, r, (W_n, n \ge 0), (X_0,Y_0)).
\]
The strategies $(u_n, n \ge 0)$ and $(v_n, n \ge 0)$ are
chosen by Alice and Bob respectively in order to maximize
the performance objective given in eqn. \eqref{eq:performance}.

Since our aim is to demonstrate the existence of a quantum advantage relative to classical strategies, proving this
for classical strategies that allow for common randomness, as above,
immediately implies that there is a quantum advantage
relative to strategies that only allow private randomization. This is because in defining 
privately randomized strategies the individual random
seeds involved in the private randomizations can be thought of as being provided to both players, while each player just ignores the random seed intended for the other player.

In contrast to the case when there is a 
centralized controller, there is no broadly 
applicable general theory that allows one to 
determine the optimal strategies of Alice and
Bob in problems of this kind. It is clear that
what one needs to come to grips with is the beliefs of each of the controllers about what the other controller believes, but this is
a hierarchical construct, which is not tractable.
For instance Alice has a belief (i.e. 
a conditional probability distribution, given her observations) over the state
(i.e. the pair $(X_n, Y_n)$) at time $n$, as does
Bob. Alice would then need to maintain a belief
about Bob's belief about the state, as would
Bob about Alice's belief about the state, but then 
Alice would need to maintain a belief about Bob's belief about her belief about Bob's belief about the state, and so on. To the best of this author's knowledge, nothing of broad applicability that allows one to penetrate this thicket of beliefs has been discovered in the work so far on decentralized control problems of this nature except, of course, if one imposes various kinds of restrictive assumptions on the underlying dynamics.

Nevertheless, we will show, by example, that 
the availability of quantum entanglement between Alice and Bob at a fixed rate can result in strictly improved performance in problems of this kind. 
We will do this in the context of a specific example,
which is introduced in next section.

\section{A specific example of a decentralized 
POMDP
%partially-observed Markov decision processes 
%with two controllers
}   \label{sec:examplemodel}

We restrict attention 
%in this paper 
now
to a specific example of a decentralized 
POMDP
%partially-observed Markov decision process 
with 
two controllers fitting the general model of the preceding section. The existence of a quantum advantage in the
decentralized control of POMDP will be demonstrated in the 
context of this example.

Specifically, we take $\mcX = \mcY = \{1,2,3\}$. We let
\[
\mcU = \{u = (u^{(1)}, u^{(2)}, u^{(3)}) : u^{(l)} \in \{1, -1\} \mbox{ for all $l \in \{1, 2, 3\}$ and $\prod_{l=1}^3 u^{(l)} = 1$}\},
\]
and we let 
\[
\mcV = \{v = (v^{(1)}, v^{(2)}, v^{(3)}) : v^{(k)} \in \{1, -1\} \mbox{ for all $k \in \{1, 2, 3\}$ and $\prod_{k=1}^3 v^{(k)} = -1$}\}.
\]
Let $\frac{1}{9} > \delta > 0$. The Markovian kernel
\[
q(x_{n+1}, y_{n+1}| x_n, y_n, u_n, v_n),
\]
is some fixed kernel, with the only requirement being that 
\begin{equation}        \label{eq:deltabound}
q(i,j| x_n, y_n, u_n, v_n) > \delta \mbox{ for all $(x_n, y_n, u_n, v_n) \in \mcX \times \mcY \times \mcU \times \mcV$}.
\end{equation}
The reward function 
%\[
%r: \mcX \times \mcY \times \mcU \times \mcV \to \mbbR,
%\]
is given by
\begin{equation}        \label{eq:rewards}
r(i,j,u,v) = u^{(j)} v^{(i)}.
\end{equation}

Together with the initial condition, described by the 
random pair $(X_0,Y_0)$, and the structure of the 
common randomness, described by the sequence of indepednent
random variables $(W_n, n \ge 0)$ (which are also independent of $(X_0,Y_0)$) this completely describes a
specific decentralized POMDP
\[
(\mcX, \mcY, \mcU, \mcV, q, r, (W_n, n \ge 0), (X_0,Y_0)).
\]

\subsection{An upper bound on performance
with classical strategies
%in the example on the long term average reward with classical strategies and common randomness
}   \label{sec:classical}

A simple coupling argument establishes that
for the decentralized 
POMDP
%partially-observed Markov decision process 
%presented in the preceding section 
under consideration
we will have
\begin{equation}        \label{eq:upperbound}
\limsup_{N \to \infty}
\frac{1}{N} \sum_{n=0}^{N-1} 
E[ r(X_n, Y_n, U_n, V_n)] \le 1 - 2 \delta,
\end{equation}
for {\em all} (classical) control strategies 
(even in the presence of an arbitrary amount of common randomness between Alice and Bob, as in our formulation
of the control problem). This argument depends on the 
assumption made on the transition probabilities in 
equation \eqref{eq:deltabound}.

In our formulation a strategy is given by 
the family $(u_n(X_{0:n}, W_{0:n}), n \ge 0)$ determining the action of Alice at each time and 
the family $(v_n(Y_{0:n}, W_{0:n}), n \ge 0)$ determining the action of Bob at each time.
%Here $Z$ denotes the common randomness provided to Alice and Bob. 
Let us relax the notion of a strategy to allow each 
player to have access to the past observations of 
the other player. Thus, we now consider control strategies of the
form 
$\tilU_n := \tilu_n(X_{0:n}, Y_{0:n-1}, W_{0:n}), n \ge 0$, 
giving the action of Alice at each time and 
$\tilV_n := \tilv_n(X_{0:n-1}, Y_{0:n}, W_{0:n}), n \ge 0$, giving the action of Bob at each time. Clearly the performance objective achievable by the players with relaxed strategies
of this kind is can be no worse that achievable with 
strategies as originally defined. 
Therefore if we prove that 
\begin{equation}        \label{eq:upperboundrelaxed}
\limsup_{N \to \infty}
\frac{1}{N} \sum_{n=0}^{N-1} 
E[ r(X_n, Y_n, \tilU_n, \tilV_n)] \le 1 - 2 \delta,
\end{equation}
holds for all relaxed control strategies of Alice and Bob, then we will have proved that the inequality in eqn. \eqref{eq:upperbound}
holds for all control strategies of Alice and Bob.

To prove the inequality in eqn. \eqref{eq:upperboundrelaxed}
holds, it suffices to prove that for each $n \ge 0$
we have 
\[
E[ r(X_n, Y_n, \tilU_n, \tilV_n)] \le 1 - 2 \delta.
\]
This is an immediate consequence of Corollary \ref{cor:onenegCR} in Appendix \ref{app:aux}.
To see this write
\begin{equation}
\aligned
E[ r(X_n, Y_n, \tilU_n, \tilV_n)]
&=
E[ r(X_n, Y_n, \tilu_n(X_{0:n}, Y_{0:n-1}, W_{0:n}), \tilv_n(X_{0:n-1}, Y_{0:n}, W_{0:n}))]\\
&=
E[ E[r(X_n, Y_n, \tilu_n(X_{0:n}, Y_{0:n-1}, W_{0:n}), \tilv_n(X_{0:n-1}, Y_{0:n}, W_{0:n}))|\\
&~~~~~~~~~~~~~~~~~~~~~~~~~~~~~~~~~~~~~~~~~X_{0:n-1}, Y_{0:n-1},
W_{0:n}]]\\
&\stackrel{(a)}{=}
E[ E[r(X_n, Y_n, \tilu_n(X_n, Z_n), \tilv_n(Y_n, Z_n))|Z_n]]\\
&\stackrel{(b)}{=}
E[E[ \tilu_n^{(Y_n)}(X_n, Z_n) \tilv_n^{(X_n)}(Y_n, Z_n)| Z_n]]\\
&\stackrel{(c)}{\le}
1 -2 \delta.
\endaligned
\end{equation}
Here in step (a) we have used the notation $Z_n$ 
for the triple $(X_{0:n-1}, Y_{0:n-1}, W_{0:n})$;
in step (b) we have used the definition in 
eqn. \eqref{eq:rewards} for the reward function in the
example under consideration; and in step (c) we
have used Corollary \ref{cor:onenegcor}
in Appendix \ref{app:aux}, which tells us that
\[
P(\tilu_n^{(Y_n)}(X_n, Z_n) \tilv_n^{(X_n)}(Y_n, Z_n) = -1| Z_n)
\ge \delta.
\]
This concludes the proof of an upper bound on the achievable 
performance with classical control strategies,
even in the presence of an arbitrary amount of common randomness between Alice and Bob, in the example under consideration.

\subsection{Achieving quantum advantage 
%in the example 
with the Mermin-Peres square}   \label{sec:mermin-peres}

This section will use language that is standard in the
study of quantum mechanics and, more specifically, quantum information.
For an introduction to the basics of quantum information
and the phenomenon of quantum entanglement in product 
quantum systems, see Appendix \ref{app:quantum}.
Further, this section will refer to the Mermin-Peres square,
which is discussed in Appendix \ref{app:mpsquare}.

Consider now the decentralized POMDP of our example, 
but assume that at each time $n \ge 0$ Alice and Bob are
provided with two pairs of entangled qubits, denoted
$\rho_n(1)$ and $\rho_n(2)$. More specifically, each 
$\rho_n(m)$, for $n \ge 0$ and $m \in \{1,2\}$, is of the
form
\[
\rho_n(m) = \frac{1}{\sqrt{2}} \ket{00} + \frac{1}{\sqrt{2}} \ket{11} \in \mbbC^2 \otimes \mbbC^2,
\]
and Alice is provided with the first component, while
Bob is provided with the second component. 
All the entangled pairs of qubits are assumed to be independent. 

To demonstrate quantum advantage in our example, it is 
not necessary for us to engage with the most general
definition of strategies for Alice and Bob in this context
(which would in general allow a measurement to be carried
out at each time $n$ by Alice, based on the 
common randomness received up to that time and her
observations up to that time, on the portion of the 
system, comprised of the first components of each pair
of qubits received at each time from $0$ through $n$, 
and then act based on the outcome of this measurement;
note that some of those qubits might have already been measured in the past, and so their state might have changed based on what the outcomes of the measurements were in the past; and similarly for Bob). Rather, it suffices to restrict attention to a class of strategies for each agent that are easier to discuss: at each time Alice just measures the system comprised of the first components of the two fresh qubits received at that time and then 
acts based on the result of this measurement and 
her observations so far and the common randomness received
so far; and similarly for Bob).

To be even more specific, we will simply consider strategies of this more restricted kind
for Alice and Bob that are based on the Mermin-Peres square, which is discussed in Appendix
\ref{app:mpsquare}. 

At time $n \ge 0$ Alice ignores the
common randomness and her past observations $X_{0:n-1}$.
For $1 \le i \le 3$, if $X_n = i$ she carries out the measurements on the pair of qubits corresponding to the
first components of the entangled pairs of qubits
$\rho_n(1)$ and $\rho_n(2)$ (which she has access to)
as described by the $i$-th row of the Mermin-Peres square.
As discussed in Appendix \ref{app:mpgameandsquare}
the resulting outcomes will be in $\{1, -1\}$ and will
not depend on the order in which these measurements are carried out. For $1 \le l \le 3$, Alice 
chooses the $l$-th component of $U_n(X_n)$, i.e. $U_n^{(l)}(X_n)$, to be the 
result of the measurement corresponding to the
$(X_n,l)$ entry of the Mermin-Peres square.

Similarly, at time $n \ge 0$ Bob ignores the
common randomness and his past observations $Y_{0:n-1}$.
For $1 \le j \le 3$, if $Y_n = j$ he carries out the measurements on the pair of qubits corresponding to the
second components of the entangled pairs of qubits
$\rho_n(1)$ and $\rho_n(2)$ (which he has access to)
as described by the $j$-th column of the Mermin-Peres square.
As discussed in Appendix \ref{app:mpgameandsquare}
the resulting outcomes will be in $\{1, -1\}$ and will
not depend on the order in which these measurements are carried out. For $1 \le k \le 3$, Bob
chooses the $k$-th component of $V_n(Y_n)$, i.e. $V_n^{(k)}(Y_n)$,  to be the 
result of the measurement corresponding to the
$(k,Y_n)$ entry of the Mermin-Peres square.

As argued in Appendix \ref{app:mpgameandsquare}, 
this has the amazing consequence that 
\[
U_n^{(Y_n)}(X_n) V_n^{(X_n)}(Y_n) = 1,
\]
pointwise. Hence we will have 
\[
E[ r(X_n, Y_n, U_n, V_n)] = E[U_n^{(Y_n)}(X_n) V_n^{(X_n)}(Y_n)] = 1,
\]
so that, for this strategy aided by quantum entanglement we have
\[
\lim_{N \to \infty}
\frac{1}{N} \sum_{n=0}^{N-1} 
E[ r(X_n, Y_n, U_n, V_n)] = 1.
\]

Since $1 > 1 - 2 \delta$, this establishes the existence of a quantum advantage in 
the decentralized control of POMDPs, which was the main point of writing this paper.

\section{An example where one-shot quantum advantage exists but dynamical quantum advantage does not}   \label{sec:staticvsdynamic}

We give an example to emphasize that the existence of a quantum advantage 
at the static (one-shot) level does not imply that there is a quantum
advantage at the dynamic level. Let us first define formally what we mean by this statement.
Consider a decentralized POMDP defined by
\[
(\mcX, \mcY, \mcU, \mcV, q, r, (W_n, n \ge 0), (X_0,Y_0)),
\]
as in Section \ref{sec:generalmodel},
and where the strategies
$(u_n, n \ge 0)$ and $(v_n, n \ge 0)$ are
chosen by Alice and Bob respectively 
as in eqns. \eqref{eq:AliceU} and \eqref{eq:BobV}
respectively, 
in order to maximize
the performance objective given in eqn. \eqref{eq:performance}. 

We will say that there is no quantum advantage at the
dynamical level if the supremum of
\[
\liminf_{N \to \infty}
\frac{1}{N} \sum_{n=0}^{N-1} 
E[ r(X_n, Y_n, U_n, V_n)]
\]
over all classical strategies is the same as that over all strategies
for Alice and Bob where they are also provided with 
quantum entanglement at a fixed rate. We will say that
there is a static quantum advantage if there is an
initial probability distribution for $(X_0,Y_0)$
such that the supremum of 
\[
E[r(X_0,Y_0,U_0,V_0)]
\]
over all classical strategies (given by $U_0 = u_0(X_0,W_0)$
for Alice and $V_0 = v_0(Y_0,W_0)$ for Bob) 
is strictly smaller than this supremum when, in addition,
Alice and Bob are provided with quantum entanglement. 

With this formalism in mind, consider the following 
example to establish our claim. Once again, as in 
Section \ref{sec:examplemodel}, we have
$\mcX = \mcY = \{1,2,3\}$. We again have
\[
\mcU = \{u = (u^{1)}, u^{(2)}, u^{(3)}) : u^{(l)} \in \{1, -1\} \mbox{ for all $l \in \{1, 2, 3\}$ and $\prod_{l=1}^3 u^{(l)} = 1$}\},
\]
and 
\[
\mcV = \{v = (v^{(1)}, v^{(2)}, v^{(3)}) : v^{(k)} \in \{1, -1\} \mbox{ for all $k \in \{1, 2, 3\}$ and $\prod_{k=1}^3 v^{(k)} = -1$}\}.
\]
Further, the reward function is given by 
$r(i,j,u,v) = u^{(j)} v^{(i)}$,
as in eqn. \eqref{eq:rewards}.
However, now the Markovian kernel
is given by 
\[
q(x_{n+1}, y_{n+1}| x_n, y_n, u_n, v_n) = 
1( (x_{n+1}, y_{n+1}) = \tau((x_n, y_n)),
\]
where 
\[
\tau: \mcX \times \mcY \mapsto  \mcX \times \mcY,
\]
defines a periodic walk through the
state space that visits each state exactly once before returning to the initial state, given by the sequence
\[
\to (1,1) \to (1,2) \to (2,3) \to (2,2) \to (3,3) \to (3,1) \to (1,3) \to (2,1) \to (3,2) \to (1,1) \to,
\]
i.e. $\tau((1,1)) = (1,2)$, $\tau((1,2)) = (2,3)$, etc.

It is not hard to see that, whatever the initial condition,
within two steps each agent becomes aware not only of its
own observations but also of those of the other agent. 
Namely, for all $n \ge 2$ we have that $Y_{0:n}$ is a deterministic function of $X_{0:n}$ and vice versa.
This means that the following classical strategies 
can be implemented by the two agents for $n \ge 2$.
Alice ignores any available common randomness and 
chooses $U_n$, based on $X_n$, which she knows, such that $U_n^{(Y_n)}(X_n) = 1$ (this is possible
because, as we just argued, Alice also knows $Y_n$ when 
$n \ge 2$) and she chooses $U_n^{(l)}(X_n)$
for $1 \le l \neq Y_n \le 3$ in some way such that the constraint 
$\prod_{l=1}^3 U_n^{(l)}(X_n) = 1$ is satisfied.
Similarly Bob ignores any available common
randomness and chooses $V_n$, based on $Y_n$, which he knows, such that $V_n^{(X_n)}(Y_n) = 1$,
and $V_n^{(k)}(Y_n)$
for $1 \le k \neq X_n \le 3$ such that the constraint 
$\prod_{k=1}^3 V_n^{(k)}(Y_n) = -1$ is satisfied.
This is possible for $n \ge 2$ because Bob also knows $X_n$.
One then has 
\[
E[ r(X_n, Y_n, U_n, V_n)] = E[ U_n^{(Y_n)}(X_n) V_n^{(X_n)}(Y_n)] = 1,
\]
for all $n \ge 2$ and so, with this classical strategy,
we have
\[
\lim_{N \to \infty}
\frac{1}{N} \sum_{n=0}^{N-1} 
E[ r(X_n, Y_n, U_n, V_n)] =1.
\]
There can be no dynamical quantum advantage, since 
it is impossible to beat a long term average reward of 
$1$ given that the reward at each time is pointwise bounded
by $1$.

On the other hand, from our earlier discussion, we can 
conclude that one-shot quantum advantage exists in this
example. Indeed, suppose that the initial distribution 
of the state is uniform over all the nine possibilities.
Then, from Corollary \ref{cor:onenegCR} in Appendix
\ref{app:aux} we can conclude that no classical strategy 
can achieve an expected reward of more than $\frac{7}{9}$
(this was also discussed in detail in Appendix 
\ref{app:mpgame}). But, as seen in Appendix 
\ref{app:mpgameandsquare}, if Alice and Bob are 
provided with two pairs of entagled qubits, each 
in the state $\frac{1}{\sqrt{2}} \ket{00} +
\frac{1}{\sqrt{2}} \ket{11}$, the two pairs being independent, with Alice being provided with the first
coordinate of each pair and Bob being provided with the
second coordinate of each pair, then they can each carry
out measurements as prescribed the appropriate row (for Alice) and column (for Bob) of the Mermin-Peres square, and
can thereby achieve a one-shot reward of $1$, which is strictly bigger than $\frac{7}{9}$.

\section{Concluding remarks}        \label{sec:conclusions}

We have demonstrated via an example that the provision of
quantum entanglement at a fixed rate to two agents who are
working together to maximize a long term average reward 
criterion in a partially-observed Markov decision scenario can lead to a strict improvement in performance, i.e. a quantum advantage. The argument to show this builds on a 
well-known and astonishing example in the theory of 
quantum information, called the Mermin-Peres square. 
%As such, this appears to be the first time that quantum advantage has been demonstrated in control in  a dynamical scenario,
%although it is debatable whether this is surprising or
%not since 
While
quantum advantage is already known to exist in static team problems, in Section \ref{sec:staticvsdynamic}
we have given an example suggesting that it may be too
facile to take for granted that the existence of a quantum advantage in 
static problems implies its existence in dynamical scenarios.

%But really, the main point of writing this paper is two-fold.
%First, the purpose is bring the Mermin-Peres square to the attention to the community of control theorists interested in such issues. More importantly, it is to prompt 
This work suggests the
%an 
investigation of what seems to be a central question:
for which decentralized POMDP
\[
(\mcX, \mcY, \mcU, \mcV, q, r, (W_n, n \ge 0), (X_0,Y_0))
\]
do we have quantum advantage and for which ones do we not?
To address this question in the case where quantum 
entanglement is provided to the two agents at a fixed rate,
one should ideally work with the most general notion of adapted control strategies for the two players in the presence of quantum entanglement, which allows for repeated measurement of previously measured quantum systems.

\section*{Acknowledgements}

This research was supported by the grants
CCF-1901004 and CIF-2007965 from the U.S. National
Science Foundation. We would like to thank Ankur Kulkarni for suggesting that the paper would be stronger if one could construct an example making the point that dynamical quantum advantage may not exist even when there is a one-shot quantum advantage.

\appendix 

\section{Some auxiliary results} \label{app:aux}

In this appendix we gather some auxiliary results that are used in the main discussion.

\begin{lemma}   \label{lem:oneneg}
%Let $u = (u[1], u[2], u[3])$ and 
%$v = (v[1], v[2], v[3])$ have entries in 
%$\{1, -1\}$ and satisfy
%$u[1]u[2]u[3] = 1$ and 
%$v[1]v[2]v[3] = -1$. 
Let $u = (u^{(1)},u^{(2)},u^{(3)})$ and 
$v = (v^{(1)},v^{(2)},v^{(3)})$ have entries in 
$\{1, -1\}$ and satisfy
%$u_1 u_2 u_3 = 1$ 
$\prod_{l=1}^3 u^{(l)} = 1$
and 
%$v_1 v_2 v_3 = -1$. 
$\prod_{k=1}^3 v^{(k)} = -1$.
Then there is at least
one choice of a pair of indices $(i,j)$
with $i, j \in \{1,2,3\}$ such that 
%$u[i]v[j] = -1$. 
$u^{(j)}v^{(i)} = -1$. 
\end{lemma}

\begin{proof}
 Note that 
%$u[i]v[j] \in \{1, -1\}$
$u^{(j)}v^{(i)} \in \{1, -1\}$
for all $i, j \in \{1,2,3\}$. 
Hence suppose, to the contrary, that we have
%$u[i]v[j] = 1$
$u^{(j)}v^{(i)} = 1$
for all $i, j \in \{1,2,3\}$.
It follows that 
%$u[1]u[2]u[3]v[1]v[2]v[3] = 1$.
%$u_1 u_2 u_3 v_1 v_2 v_3 = 1$.
$(\prod_{l=1}^3 u^{(l)} )(\prod_{k=1}^3 v^{(k)}) = 1$.
But this is 
false. This concludes the proof.
\end{proof}
%\hfill $\Box$

\begin{remark}
Clearly the conclusion of Lemma 
\ref{lem:oneneg} 
can be strengthened. However, our 
overall aim
is just to make a qualitative point about decentralized control, so we do not attempt to optimize lemma statements in unnecessary ways. \hfill $\Box$
\end{remark}

The following result can be viewed as a corollary of Lemma \ref{lem:oneneg}, since it has a similar proof.

\begin{corollary}       \label{cor:onenegcor}
Fix $\delta > 0$.
Let $(X,Y) \in \{1,2,3\} \times \{1,2,3\}$ be a
pair of random variables with 
$P((X,Y) = (i,j)) \ge \delta$ for all
$i, j \in \{1,2,3\}$. Let $u(X)$ and $v(Y)$ be as
in the statement of Lemma \ref{lem:oneneg}, 
i.e. we have $\prod_{l=1}^3 u^{(l)}(X) = 1$ and 
$\prod_{k=1}^3 v^{(k)}(Y) = -1$ pointwise. Then we have
%$E[u(X)v(Y)] \le 1 - 2 \delta$.
%$E[u_Yv_X] \le 1 - 2 \delta$. 
$P(u^{(Y)}(X) v^{(X)}(Y) = -1) \ge \delta$.
\end{corollary}

\begin{proof}
We have
\begin{equation*}
\aligned
P(u^{(Y)}(X) v^{(X)}(Y) = -1)  &= 
\sum_{i=1}^3 \sum_{j=1}^3 P((X,Y) = (i,j)) 1(u^{(j)}(i) v^{(i)}(j) = -1)\\
&\ge \delta \sum_{i=1}^3 \sum_{j=1}^3 1(u^{(j)}(i) v^{(i)}(j) = -1)\\
&\ge \delta,
\endaligned
\end{equation*}
where for the last step we observe that if we were to have 
$u^{(j)}(i) v^{(i)}(j)$ equal to $1$ for all $(i,j)$ then we would have
\begin{equation*}
\prod_{i=1}^3 \prod_{j=1}^3 u^{(j)}(i) v^{(i)}(j) = 
\left( \prod_{i=1}^3 \prod_{l=1}^3 u^{(l)}(i) \right) 
\left( \prod_{j=1}^3 \prod_{k=1}^3 v^{(k)}(j) \right) = 1,
\end{equation*}
which is false, because the term in the middle should be $-1$.
\end{proof}
%\hfill $\Box$

We also have the following corollary, which can be viewed as a version of Corollary \ref{cor:onenegcor} where there is common randomness between the agents creating the components $u$ and $v$ of Lemma \ref{lem:oneneg} from the respective indices in the pair $(X,Y)$. 
%and also does not merit a formal proof.
The proof is similar to that of Corollary \ref{cor:onenegcor}
and will be omitted.

\begin{corollary}       \label{cor:onenegCR}
Fix $\delta > 0$. Let $\mcZ$ be an arbitrary 
%complete separable metric space.
Polish space.
Let $(X,Y,Z) \in \{1,2,3\} \times \{1,2,3\} \times \mcZ$ be a random triple with 
$P((X,Y) = (i,j)|Z) \ge \delta$ almost surely, for all
$i, j \in \{1,2,3\}$. 
Let 
%$u(w) = (u(w)[1], u(w)[2], u(w)[3])$ and 
%$v(w) = (v(w)[1], v(w)[2], v(w)[3])$ 
$u(X,Z) = (u^{(1)}(X,Z), u^{(2)}(X,Z), u^{(3)}(X,Z))$ and 
$v(Y,Z) = (v^{(1)}(Y,Z), v^{(2)}(Y,Z), v^{(3)}(Y,Z))$ 
be measurable, with each coordinate being a
$\{1, -1\}$-valued  function, and such that 
%of $Z$ 
%with each $u_j(Z)$
%and $v_i(Z)$ for $1 \le i,j \le 3$ being
%$\{1, -1\}$-valued and satisfying
%$u_1(Z)u_2(Z)u_3(Z) = 1$
$\prod_{l=1}^3 u^{(l)}(X,Z) = 1$ 
and 
%$v_1(Z)v_2(Z)v_3(Z) = -1$ 
$\prod_{k=1}^3 v^{(k)}(Y,Z) = 1$
almost surely.
%for all $w \in \mcW$.
Then we have 
%$E[u_Y(Z)v_X(Z)] \le 1 - 2 \delta$.
%\[
%P(u_Y(Z) v_X(Z) = -1|Z) \ge \delta.
%\]
\begin{equation}
P(u^{(Y)}(X,Z) v^{(X)}(Y,Z) = -1|Z) \ge \delta.
\end{equation}
\hfill $\Box$
\end{corollary}

%{\em Proof}: Write 
%\[
%E[u_Y(Z)v_X(Z)] = 
%E[ E[u_Y(Z)v_X(Z)|Z]] \le 1 - 2 \delta,
%\]
%where the second step comes from 
%$E[u_Y(Z)v_X(Z)|Z] \le 1 - 2 \delta$, which
%is what Corollary \ref{cor:onenegcor} tell us.
%\hfill $\Box$

\section{Quantum information}        \label{app:quantum}

%In this appendix we gather the terminology and 
%results from the theory of 
%quantum information that we need for our purposes.  
We will focus only on what
is needed to formalize the notion of quantum entaglement
between a pair of qubits, since this suffices 
to discuss the Mermin-Peres
square. For a more thorough introduction to the basics of 
quantum information we refer the reader to the textbooks \cite{Holevo12} and \cite{Watrous18}.

\subsection{A single quantum system}    \label{app:singlequant}

As usual, $\mbbC^{n \times n}$ denotes the set of 
$n \times n$ matrices with complex entries, where $n \ge 1$.
%Consider an $n$-dimensional 
%complex vector space with an inner product, which we denote by 
%$\mbbC^n$ even though
%even though the standard orthonormal basis chosen to write this
%vector space as $\mbbC^n$ is not relevant to the subsequent
%discussion.
Let $L_n$ denote the set of all linear mappings from 
$\mbbC^n$ to $\mbbC^n$, which we identify with 
$\mbbC^{n \times n}$ via the choice of the standard orthonormal basis in $\mbbC^n$.
We write $\Tr(M)$ for the trace of
%matrix $M \in \mbbC^{n \times n}$. 
the linear mapping $M \in L_n$.
For a vector $v \in \mbbC^n$ 
(thought of as a column vector), 
%(thought of as defining the linear mapping from $\mbbC$ to 
%$\mbbC^n$ that takes $1$ to $v$)
we write $v^*$ for 
%the corresponding adjoint map 
%(thus, if $v$ is thought of as a column vector in the 
%standard basis then 
%$v^*$ would be
%its complex conjugate transpose), 
%$v^*$ for 
its complex conjugate transpose,
and for any 
%$M \in \mbbC^{n \times n}$ 
$M \in L_n$
we write $M^*$ for its 
complex conjugate transpose, 
%adjoint map,
these conventions being consistent for $n=1$.
%If $M$ is written in the standard orthonormal basis as a matrix in 
%$\mbbC^{n \times n}$, then $M^*$ would be the complex conjugate
%transpose of $M$. 

Let 
%$D_n \subseteq \mbbC^{n \times n}$ 
$D_n \subseteq L_n$
denote the subset of positive-semidefinite 
matrices 
with trace $1$. Elements of
$D_n$ are called {\em density matrices}. The state of a quantum system
is described by a density matrix. Every density matrix is 
Hermitian, since this is part of what it means to be positive-semidefinite. We write 
%$\Pos_n \subseteq \mbbC^{n \times n}$ 
$\Pos_n \subseteq L_n$
for the subset of positive-semidefinite matrices, so 
%$D_n \subseteq \Pos_n \subseteq \mbbC^{n \times n}$. 
$D_n \subseteq \Pos_n \subseteq L_n$. 
Vectors of norm $1$ in $\mbbC^n$
correspond to the {\em pure states} of the quantum
system: the vector $v \in \mbbC^n$ corresponds
to the pure state $v v^*$ where $v^*$ denotes
the complex conjugate transpose of $v$. 
We will say that the quantum system is of dimension $n$ when its states are described by 
density matrices in $D_n$.

As an example, let $n =2$. The corresponding quantum system
is called a {\em qubit}.
Using Dirac notation we write
$\ket{0}$ and $\ket{1}$ for the vectors of the standard orthonormal
basis in $\mbbC^2$. Any element of $D_2$ (which one can 
identify with a positive-semidefinite matrix in $C^{2 \times 2}$)
is a state for the qubit. For instance, the state
$\begin{bmatrix} 1 & 0 \\ 0 & 0 \end{bmatrix}$ is the pure
state corresponding to the vector $\ket{0}$, 
$\begin{bmatrix} 0 & 0 \\ 0 & 1 \end{bmatrix}$ is the pure
state corresponding to the vector $\ket{1}$, and 
$\begin{bmatrix} \frac{1}{2} & \frac{1}{2} \\ \frac{1}{2} & \frac{1}{2} \end{bmatrix}$ is the pure
state corresponding to the vector $\frac{1}{\sqrt{2}} \ket{0}
+ \frac{1}{\sqrt{2}} \ket{1}$. The state 
$\begin{bmatrix}  \frac{1}{4} & \frac{\imath}{4} \\ \frac{-\imath}{4} & \frac{3}{4} \end{bmatrix}$ is not pure.

\subsection{Measurements}       \label{app:measurement}

Let $\mcA$ be a finite set. 
By a {\em measurement} we mean a map of the 
form $\mu: \mcA \mapsto \Pos_n$, with the property that
$\sum_{a \in \mcA} \mu(a) = I$, where $I$ denotes
the identity 
%matrix in $\mbbC^{n \times n}$. 
mapping in $L_n$.
Such a measurement is also called a 
positive operator-valued measurement (POVM).
The basic ansatz of quantum mechanics is that
carrying out the measurement $\mu$ on a quantum 
system in state $\rho$ results in observing 
$a \in \mcA$ with the probability $\Tr(\mu(a) \rho)$.
Carrying out the measurement also results in a change of state, depending on which $a \in \mcA$ was observed and
indeed on how the measurement was implemented,
but this is of no interest to us in this paper, so we will not discuss it. The intuitive picture
that suffices for us corresponds to the case where
each $\mu(a)$ is a {\em projection}, i.e. when we have
$\mu(a)^2 = \mu(a)$ for all $a \in \mcA$.
Such a measurement is called a projection-valued measurement 
(PVM).
After carrying out a PVM $\mu$, if the outcome 
is $a \in \mcA$, the quantum system is left
in the state $\frac{\mu(a) \rho \mu(a)}{\Tr(\mu(a) \rho)}$.
Recalling that $\mu(a)$ is a projection, it can be checked that this expression defines a density matrix.

Assuming that $\mcA$ is a subset of $\mbbR$
(or identifying $\mcA$ with such a subset) 
the PVM $\mu$ gives rise to the
Hermitian matrix $\sum_{a \in \mcA} a \mu(a)$.
With this in mind, it is customary to think of
every Hermitian matrix
as giving rise to the PVM 
(with $\mcA$ a subset of $\mbbR$) defined by its spectral 
decomposition based on the eigenspaces corresponding to 
its distinct eigenvalues.
\iffalse
For example
\[
\begin{bmatrix} \frac{7}{8} & 0 \\ 0 & 20 
\end{bmatrix} = \frac{7}{8} \begin{bmatrix} 1 & 0 \\ 0 & 0 \end{bmatrix}
+ 20 \begin{bmatrix} 0 & 0 \\ 0 & 1 \end{bmatrix}
\]
defines a PVM on qubits. When this measurement is carried out
on the qubit in state $\begin{bmatrix}  \frac{1}{4} & \frac{\imath}{4} \\ \frac{-\imath}{4} & \frac{3}{4} \end{bmatrix}$
it results in the outcome $\frac{7}{8}$ with probability 
$\frac{1}{4}$ and the outcome $20$ with probability 
$\frac{3}{4}$. In the former case the measurement leaves the
qubit in state $\begin{bmatrix} 1 & 0 \\ 0 & 0 \end{bmatrix}$,
while in the latter case it leaves the qubit in state
$\begin{bmatrix} 0 & 0 \\ 0 & 1 \end{bmatrix}$.
\fi
For example, the Hermitian matrix
\begin{equation}    \label{eq:sigmaxdecomp}
\sigma_x := \begin{bmatrix} 0 & 1 \\ 1 & 0 
\end{bmatrix}
=  \begin{bmatrix} \frac{1}{2} & \frac{1}{2} \\ \frac{1}{2} & 
\frac{1}{2}
\end{bmatrix} -
 \begin{bmatrix} \frac{1}{2} & - \frac{1}{2} \\ - \frac{1}{2} & 
\frac{1}{2}
\end{bmatrix}
\end{equation}
can be thought of as defining a PVM 
$\mu: \{1, -1\} \mapsto D_2$ on qubits, given by
\begin{equation}    \label{eq:sigmaxfactors}
\mu(1) = \begin{bmatrix} \frac{1}{2} & \frac{1}{2} \\ \frac{1}{2} & 
\frac{1}{2}
\end{bmatrix},~
\mu(-1) = 
\begin{bmatrix} \frac{1}{2} & - \frac{1}{2} \\ - \frac{1}{2} & 
\frac{1}{2}
\end{bmatrix}.
\end{equation}
When this measurement is carried out
on the qubit in state $\rho := \begin{bmatrix}  \frac{1}{4} & \frac{\imath}{4} \\ \frac{-\imath}{4} & \frac{3}{4} \end{bmatrix}$
it results in the observing $1$ with probability 
$\Tr( \mu(1) \rho) = \frac{1}{2}$ and observing 
$-1$ with probability 
$\Tr( \mu(-1) \rho) = \frac{1}{2}$.
In this example, after the measurement the qubit is left in the
state $\mu(1)$ if the outcome is $1$ and in the state $\mu(-1)$
if the outcome is $-1$ (in general one needs to use the formula
$\frac{\mu(a) \rho \mu(a)}{\Tr(\mu(a) \rho)}$ to figure out the post-measurement state of a PVM; what happens in this example is special because each $\mu(a)$ is of rank $1$).

It can be checked that if two Hermitian matrices commute then, when each is viewed as a measurement, it does not matter in what
order the two measurements are performed in the sense that for either order of performing the measurements
the joint probability distribution of the pair of outcomes will be the same, and the state in which the system is left after the two measurements,
given the respective outcomes, is the same in both cases.

\subsection{Pauli matrices}       \label{app:Pauli}

This is a good point at which to introduce the Pauli matrices,
which are central to the understanding of the 
Mermin-Peres square. 
There are four Pauli matrices,
each of which is a Hermitian matrix in $\mbbC^{2 \times 2}$, 
namely
\[
\sigma_0 := I,~
\sigma_x := \begin{bmatrix} 0 & 1 \\ 1 & 0 
\end{bmatrix},~
\sigma_y = \begin{bmatrix} 0 & - \imath \\ \imath & 0 
\end{bmatrix},~
\sigma_z = \begin{bmatrix} 1 & 0 \\ 0 & -1 
\end{bmatrix}.
\]
It can be checked that these matrices obey the following multiplication rule:
\begin{center}
\begin{tabular}{|c|c|c|c|}
\hline
 & $\sigma_x$ & $\sigma_y$ & $\sigma_z$ \\
 \hline
 $\sigma_x$ & $\sigma_0$ & $\imath \sigma_z$ & $- \imath \sigma_y$ \\
 \hline
 $\sigma_y$ & $- \imath \sigma_z$ & $\sigma_0$ & $\imath \sigma_x$ \\
 \hline
 $\sigma_z$ & $\imath \sigma_y$ & $- \imath \sigma_x$ & $\sigma_0$ \\
 \hline
\end{tabular}
\end{center}
where the row labels are in the first column, the column labels
are in the first row, and each of the other entries represents
the multiplication of the row index followed by column index,
e.g. $\sigma_y \sigma_z = \imath \sigma_x$.

The Pauli matrix $\sigma_0$ has the unique eigenvalue $1$, while each of the other three
has eigenvalues $1$ and $-1$. Thus, when a Pauli matrix
is viewed as a measurement on a qubit, the observation will always be $1$ for $\sigma_0$
and will be either $1$ or $-1$ in each of the other three cases. The probability of the 
observation will depend on the state of the qubit being measured in each of the three
nontrivial cases, but it can be checked that
the post-measurement state of the qubit depends only on the observation and not on the 
pre-measurement state in each of these cases (on the other hand, the post-measurement 
state equals the pre-measurement state in case the measurement $\sigma_0$ is carried out).

\subsection{Products of quantum systems}       \label{app:product}

Given two quantum systems of dimensions
$m$ and $n$ respectively, the joint system is
of dimension $mn$. 
A state $\rho \in D_{mn}$ of the joint
system can be thought of as element of
$C^{mn \times mn}$ by the choice of the standard orthonormal
basis in $\mbbC^{mn}$.
%The pure states of the joint
%system can be identified with the element of
Recall that the tensor product $\mbbC^n \otimes \mbbC^m$
can be identified with $\mbbC^{mn}$. 
Recall also that given $A \in L_n$ and
$B \in L_m$, their tensor product 
$A \otimes B$ can be viewed as an element of 
$L_{mn}$. In matrix terms, the
$((i,k), (j,l))$ entry of $A \otimes B$ is 
$a(i,j) b(k,l)$, where 
the entries of $A$ are denoted $a(i,j)$
those of $B$ are denoted
$b(k.l)$ and where in $A \otimes B$ the rows and columns
are listed in lexicographic order.
Of course, 
$A \otimes B \in L_n \otimes L_m$, but recall that
$L_n \otimes L_m$ is naturally identified with 
$L_{mn}$, because the notation $L_n \otimes L_m$
encompasses all linear combinations (with coefficients in 
$\mbbC$) of elements of the form $A \otimes B$ where
$A \in L_n$ and $B \in L_m$. 

Not every element of $L_{mn}$ can be expressed in the
form $A \otimes B$ where $A \in L_n$ and $B \in L_m$.
If the product system is in a state $\rho \in D_{mn}$ which can be written in the 
form $\rho_A \otimes \rho_B$ where $\rho_A \in D_n$ and
$\rho_B \in D_m$ then the component systems are said to be {\em independent} (in this overall state), and the state itself is called a {\em product state}.
It can be checked that the use of the term ``independent"
in this sense is consistent with its use in classical 
probability theory (i.e. when the states involved 
are diagonal matrices with nonnegative entries and trace $1$).

%It is not hard to verify that there is no difficulty in 
%extending this approach in the natural way to define the product
%system of an arbitrary finite number of quantum systems. 

As an example of the kind of calculations needed to understand the Mermin-Peres square, 
consider the product of two qubit systems. This is a $4$-dimensional 
system, which can be described in matrix notation by the choice of the 
basis $\{ \ket{00}, \ket{01}, \ket{10}, \ket{11}\}$ for 
$\mbbC^2 \otimes \mbbC^2$,
%Note that the product system is $4$-dimensional
%(because two times two is four) 
%and the canonical 
%orthonormal basis for it is $\{e_{00}, e_{01}, e_{10}, e_{11}\}$, where $e_{ij}$ denotes 
%$e_i \otimes e_j$ for $i,j \in \{0,1\}$.
where $\ket{ij}$ denotes $\ket{i} \otimes \ket{j}$ for $i,j \in \{0,1\}$.
As an example, the Hermitian matrix $\sigma_x \otimes \sigma_y$
can be thought of as a measurement on this product system (this measurement has two possible outcomes,
i.e. $1$ or $-1$). Similar to the way that we wrote
$\sigma_x = \mu(1) - \mu(-1)$ in the notation of
eqns. \eqref{eq:sigmaxdecomp} and \eqref{eq:sigmaxfactors}, we can write 
$\sigma_y = \nu(1) - \nu(-1)$, where 
\begin{equation}    \label{eq:sigmaydecomp}
\sigma_y := \begin{bmatrix} 0 & - \imath \\ \imath & 0 
\end{bmatrix}
=  \begin{bmatrix} \frac{1}{2} & - \frac{\imath}{2} \\ \frac{\imath}{2} & 
\frac{1}{2}
\end{bmatrix} -
 \begin{bmatrix} \frac{1}{2} & \frac{\imath}{2} \\ - \frac{\imath}{2} & 
\frac{1}{2}
\end{bmatrix}
\end{equation}
corresponding to PVM 
$\nu: \{1, -1\} \mapsto D_2$ on qubits, given by
\begin{equation}    \label{eq:sigmayfactors}
\nu(1) = \begin{bmatrix} \frac{1}{2} & - \frac{\imath}{2} \\ \frac{\imath}{2} & 
\frac{1}{2}
\end{bmatrix},~
\nu(-1) = 
\begin{bmatrix} \frac{1}{2} &  \frac{\imath}{2} \\ - \frac{\imath}{2} & 
\frac{1}{2}
\end{bmatrix}.
\end{equation}
Thus $\sigma_x \otimes \sigma_y$ can be thought of
as corresponding to the PVM
\[
\beta: \{1,-1\} \mapsto D_4
\]
given by 
\[
\beta(1) = \mu(1) \otimes \nu(1) + \mu(-1) \otimes \nu(-1) \mbox{ and }
\beta(-1) = \mu(1) \otimes \nu(-1) + \mu(-1) \otimes \nu(1).
\]
Suppose now that we carry out the measurement corresponding to $\sigma_x \otimes \sigma_y$ on the
pure state in $D_4$ coresponding to the vector
\[
\frac{1}{2} \ket{00} + \frac{1}{2} \ket{01} 
+ \frac{1}{2} \ket{10} + \frac{1}{2} \ket{11} 
= (\frac{1}{\sqrt{2}} \ket{0} + \frac{1}{\sqrt{2}} \ket{1}) \otimes 
(\frac{1}{\sqrt{2}} \ket{0} + \frac{1}{\sqrt{2}} \ket{1}) \in \mbbC^2 \otimes \mbbC^2 = \mbbC^4.
\]
We can compute that the
outcome of this measurement will be $1$ with probability
$\frac{1}{2}$, and will be $-1$ with probability
$\frac{1}{2}$.
Writing $u$ for $\frac{1}{\sqrt{2}} \ket{0} + \frac{1}{\sqrt{2}} \ket{1} \in \mbbC^2$, we can compute that,
conditioned on the outcome being $1$, the overall
$4$-dimensional system will end up in the 
pure state corresponding to the 
vector $u \otimes \left( \frac{1+\imath}{2} \ket{0} + \frac{1 - \imath}{2} \ket{1} \right)$, while conditioned on the outcome being $-1$ it will end up in the pure state corresponding to the vector 
$u \otimes \left( \frac{1-\imath}{2} \ket{0} + \frac{1 + \imath}{2} \ket{1} \right)$.

\subsection{Entanglement}       \label{app:entanglement}

We now discuss the concept of {\em entanglement}, 
which is the extraordinary feature of quantum information that
enables the magic of the Mermin-Peres square, and hence its consequences for strict improvement of 
performance in decentralized control as discussed in this paper.

We start with a simple fact about joint probability distributions.
Suppose $\mcX$ and $\mcY$ are finite sets and $(p(x,y), (x,y) \in \mcX \times \mcY)$ is a probability distribution on
$\mcX \times \mcY$. Then, for some $L \ge 1$, there exists a probability distribution $(q_l, 1 \le l \le L)$ and
probability distributions $(a_x^{(l)}, x \in \mcX)$ 
%on $\Sigma_\mcX$ 
and 
probability distributions $(b_y^{(l)}, y \in \mcY)$ 
%on $\Sigma_\mcY$ 
such that, for all $(x,y) \in \mcX \times \mcY$, we have
\[
p(x,y) = \sum_{l=1}^L q_l a_x^{(l)} b_y^{(l)}.
\]
Indeed, there is an obvious and simple way to accomplish this by taking $L = nm$ where
$n = |\mcX|$ and $m = |\mcY|$.

This simple fact can be phrased as follows: any joint probability distribution on $\mcX \times \mcY$ is a 
convex combination of product probability distributions. This can be interpreted as a property that 
every joint probability distribution on a ``product system" needs to satisfy in the world of classical probability 
distributions. Here we think of $\mcX \times \mcY$ as being the state space of a classical ``product system"
comprised of the individual classical ``component systems" having state spaces $\mcX$ and $\mcY$ respectively.

It is now natural to ask if, in the framework of quantum information,
it holds in general that any density matrix of the product system can be expressed as a convex combination of tensor products  of density matrices (i.e. as a convex combination of
density matrices in $D_{mn}$ which can each be 
written as a tensor product of a density matrix 
in $D_n$ with one in $D_m$). 
Any density matrix in $D_{mn}$ which 
admits of a representation as such a convex combination is called {\em separable}. Any density matrix in $D_{mn}$ that is 
not separable is called {\em entangled}. 
\footnote{ Note that the notion of separability is
not an intrinsic property of a density matrix of the 
product system when the product system is viewed as just a system. It only makes sense when the product system is viewed as a product system. Namely, we are not just discussing 
$\mbbC^{mn}$ as a complex vector space of dimension $nm$;
%where $n = |\Sigma_\mcX|$ and $m = |\Sigma_\mcY|$; 
rather, we discussing it with its explicit product
structure in terms of its specified component systems
%based on $\mcX$ and $\mcY$ respectively.
when $\mbbC^{mn}$ is identified with $\mbbC^n \otimes \mbbC^m$.
Thus the discussion of entanglement only makes sense 
in the context of the way we choose to think of the
product system as having been created from specified component systems. Indeed, a density matrix of a system can be entangled for some particular way of writing that system as a product system while being not entangled, i.e. separable, when it is thought of in terms of some other way of writing the system as a product system.}

\subsection{Existence of entanglement}    \label{app:entanglementexample}

The heart of the matter is that there are entangled density matrices (i.e. states) in product systems. 
For an example, which is the one used in the discussion of the Mermin-Peres square, let us take $n = m = 2$ (i.e. the component systems
are qubits). We will show that the density matrix
$vv^*$ corresponding to 
\[
v := \frac{1}{\sqrt{2}} \ket{00} + \frac{1}{\sqrt{2}} \ket{11}
\]
in the product system is entangled. 
%Here we are using the Dirac notation $|00>$ for the 
%canonical basis vector $e_{00}$ in the product system,
%and similarly $|11>$ for $e_{11}$, the reason for doing this being that this is the way you will usually find this example discussed in the published literature.
%(recall that
%the first system is based on $\mcX$ and the second system on $\mcY$). 
Note that we have 
\[
vv^* = \begin{bmatrix}
\frac{1}{2} & 0 & 0 & \frac{1}{2}\\
0 & 0 & 0 & 0\\
0 & 0 & 0 & 0 \\
\frac{1}{2} & 0 & 0 & \frac{1}{2}
\end{bmatrix},
\]
and we want to show that it is impossible to write
\[
vv^* \stackrel{?}{=} \sum_{l=1}^L q_l \rho_A^{(l)} \otimes \rho_B^{(l)},
\]
where $(q_l, 1 \le l \le L)$ is a probability 
distribution \footnote{We can assume without loss of generality that all the $q_l$ are strictly positive.}
and where each $\rho_A^{(l)}$ is qubit density matrix 
and each $\rho_B^{(l)}$ is a qubit density matrix. 

Since every density matrix is a convex combination
of pure states, it is equivalent to show that it is
impossible to write
\[
vv^* \stackrel{?}{=} \sum_{m=1}^M r_m 
u_A^{(m)} (u_A^{(m)})^* \otimes v_B^{(m)} (v_B^{(m)})^*,
\]
where $(r_m, 1 \le m \le M)$ is a probability 
distribution \footnote{We can assume without loss of generality that all the $r_m$ are strictly positive.}
and the $u_A^{(m)}$ and $v_B^{(m)}$ are unit
vectors in $\mbbC^2$. 

Suppose this were possible. Write
\[
u_A^{(m)} = \alpha_0^{(m)} \ket{0} + \alpha_1^{(m)} \ket{1},
\]
and
\[
v_B^{(m)} = \beta_0^{(m)} \ket{0} + \beta_1^{(m)} \ket{1},
\]
%where $|0>$ denotes $e_0$ and $|1>$ denotes $e_1$ in Dirac notation, and 
where, for each $1 \le m \le M$,
the coefficients $\alpha_0^{(m)}, \alpha_1^{(m)},
\beta_0^{(m)}, \beta_1^{(m)}$ are complex numbers 
satisfying
\[
|\alpha_0^{(m)}|^2 + |\alpha_1^{(m)}|^2 =1 
\mbox{  and  }~~
|\beta_0^{(m)}|^2 + |\beta_1^{(m)}|^2 =1.
\]
Note that the 
%$((0,1), (0,1))$ 
$(\ket{01}, \ket{01})$
entry of
$vv^*$ is $0$, so we must have
\[
\sum_{m=1}^M r_m |\alpha_0^{(m)}|^2 |\beta_1^{(m)}|^2 = 0,
\]
from which it follows that for each $1 \le m \le M$
we either have $\alpha_0^{(m)} = 0$ or
$\beta_1^{(m)} = 0$ (or both). But the 
%$((0,0),(1,1))$
$(\ket{00}, \ket{11})$
entry of $vv^*$ needs to be $\frac{1}{2}$, and this
condition turns out to be the same as
\[
\sum_{m=1}^M r_m \alpha_0^{(m)} (\alpha_1^{(m)})^* \beta_0^{(m)} (\beta_1^{(m)})^* = \frac{1}{2}.
\]
This is a contradiction and so this establishes the claimed impossibilty. We have shown that the phenomenon of entanglement exists and, more 
%\footnote{Experimental verification of entanglement in nature was one of the contributions for which the 2022 Nobel prize in Physics was awarded to Alain Aspect, John Clauser, and Anton Zeilinger.}
specifically, that the pure state $v v^*$ corresponding to 
$v = \frac{1}{\sqrt{2}} \ket{00} + \frac{1}{\sqrt{2}} \ket{11}$ in the product of two qubit systems is entangled. This observation about entanglement is all that we need for the purposes of this paper.

\section{The Mermin-Peres square}        \label{app:mpsquare}

The Mermin-Peres square \cite[Sec. 3.2.2]{Holevo12} is the following $3 \times 3$ array:

\begin{center}
\begin{tabular}{|c|c|c|} 
\hline
$\sigma_0 \otimes \sigma_z$ & $\sigma_z \otimes \sigma_0$ & $\sigma_z \otimes \sigma_z$\\
\hline
$\sigma_x \otimes \sigma_0$ & $\sigma_0 \otimes \sigma_x$ & $\sigma_x \otimes \sigma_x$ \\
\hline
$- \sigma_x \otimes \sigma_z$ & $- \sigma_z \otimes \sigma_x$ & $\sigma_y \otimes \sigma_y$\\
\hline
\end{tabular}
\end{center}

Each entry is a Hermitian matrix in $\mbbC^2 \otimes \mbbC^2$, viewed as a measurement on states in $D_4$ (which is viewed as a 
subset of $L_2 \otimes L_2$). It can be checked that each of these Hermitian matrices has eigenvalues in $\{1, -1\}$.
It can be checked that in each row $i \in \{1, 2, 3\}$ the three such Hermitian matrices in the locations $(i,1), (i,2), (i,3)$
commute with each other, and in each column $j \in \{1, 2, 3\}$ the three such Hermitian matrices in the locations $(1,j), (2,j), (3,j)$
commute with each other.

\subsection{The Mermin-Peres game}    \label{app:mpgame}

The Mermin-Peres square reveals its magic in the so-called Mermin-Peres game \cite[Sec. 3.2.2]{Holevo12}.
The game is cooperative in the sense that either both Alice and Bob win or both Alice and Bob lose.
Let Alice be the row player and
Bob the column player. Alice and Bob receive indices $i$ and $j$ respectively, chosen independently and uniformly
over $i,j \in \{1, 2, 3\}$. Alice does not know Bob's index and Bob does not know Alice's index. Alice is required to place a number
$a_{il} \in \{1, -1\}$ in each column $l \in \{1,2,3\}$, and Bob is required to place a number
$b_{kj} \in \{1, -1\}$ in each row $k \in \{1,2,3\}$. The constraint on Alice is that $\prod_{l=1}^3 a_{il} = 1$, and the 
constraint on Bob is that $\prod_{k=1}^3 b_{kj} = -1$. Alice and Bob win if $a_{ij} b_{ij} = 1$. 

If one restricts oneself to classical strategies then, even with an arbitrary amount of common randomness between Alice and Bob
(this common randomness being independent of the choices of the
indices revealed to Alice and Bob respectively) the overall probability of winning has to be strictly less than $1$. This is because, whatever
the realization (based on the common randomness) of the strategies of Alice and Bob, we must have 
$\prod_{i=1}^3 \prod_{l=1}^3 a_{il} = 1$ and $\prod_{k=1}^3 \prod_{j=1}^3 b_{kj} = -1$. Thus it is impossible to 
have $a_{ij} b_{ij} = 1$ for each choice of $i,j \in \{1,2,3\}$, which would be necessary if winning were to occur with probability $1$. 
Indeed, there must be at least one pair $(i,j)$ for which we have $a_{ij} b_{ij} = -1$ on this realization; 
see Lemma \ref{lem:oneneg} in Appendix \ref{app:aux}
for a formal proof of this obvious fact. 
From this we can conclude
that with classical strategies Alice and Bob cannot manage an overall probability of winning of more than $\frac{8}{9}$.

\subsection{The Mermin-Peres square in the Mermin-Peres game}    \label{app:mpgameandsquare}

Now suppose Alice and Bob are provided with two pairs of entangled qubits. The first pair is in the 
product state 
\[
\rho(1) = \frac{1}{\sqrt{2}} \ket{00} + \frac{1}{\sqrt{2}} \ket{11} \in \mbbC^2 \otimes \mbbC^2.
\]
Here Alice is provided with the first component and 
Bob with the second component. The second pair is in the product state 
\[
\rho(2) = \frac{1}{\sqrt{2}} \ket{00} + \frac{1}{\sqrt{2}} \ket{11} \in \mbbC^2 \otimes \mbbC^2.
\]
Here also Alice is provided with the first component and 
Bob with the second component. The overall product state is 
\[
\rho(1) \otimes \rho(2) \in (\mbbC^2 \otimes \mbbC^2) \otimes (\mbbC^2 \otimes \mbbC^2),
\]
i.e. the two entangled qubit pairs are independent
(see Appendix \ref{app:product} for the definition of independence in this context).
Note that Alice has access to the first and third factors, while Bob has access to the second and fourth factors of this overall
quantum product state.

Consider now the following strategies for Alice and Bob. On receiving the row index $i$, Alice, for each $l \in \{1,2,3\}$,
carries out the measurement given by the $(i,l)$ entry of the Mermin-Peres square on her pair state (i.e. the pair qubit 
comprised of the first and the third 
components of the overall product state). Since the three entries in that row all commute with each other, it does not matter in what
order these measurements are performed. The outcome of each measurement is in $\{1, -1\}$ and Alice writes the corresponding outcome
in the corresponding column of the row $i$. It can be checked that these three measurements satisfy the constraint on Alice 
(i.e. their product will always be $1$). Similarly, on receiving the column index $j$, Bob, for each $k \in \{1,2,3\}$ 
carries out the measurement given by the $(k,j)$ entry of the Mermin-Peres square on her pair state (i.e. the pair qubit comprised of
the second and the fourth 
components of the overall product state). Since the three entries in that column all commute with each other, it does not matter in what
order these measurements are performed. The outcome of each measurement is in $\{1, -1\}$ and Bob writes the corresponding outcome
in the corresponding row of the column $j$. It can be checked that these three measurements satisfy the constraint on Bob
(i.e. their product will always be $-1$). The incredible thing is that, with these strategies, we will have, for each 
$i,j \in \{1, 2, 3\}$ that the product of the outcome of Alice in column $j$ of row $i$ and the outcome of Bob in row $i$ of column $j$
will always be $1$. Hence the winning probability of Alice and Bob in the Mermin-Peres game becomes $1$ if they are provided with
two pairs of entangled qubits as above and then use the strategies based on the Mermin-Peres square, as just described.

\end{document}